\documentclass[11pt,a4paper,twoside]{amsart}

\usepackage{amsfonts} 
\usepackage{amssymb}
\usepackage{amsmath, amsthm}
\usepackage{mathrsfs} 
\usepackage{latexsym}
\usepackage{enumerate}
\usepackage{marginnote}
\usepackage{multicol}
\usepackage{verbatim}
\usepackage{latexsym}
\usepackage{url}

\newtheorem{thm}{Theorem}[section]
\newtheorem{lem}[thm]{Lemma}

\newtheorem{cor}[thm]{Corollary}

\theoremstyle{remark}

\newtheorem{remark}[thm]{Remark}

\def\1{1\!\!1}

\newcommand{\Ocal}{\mathcal{O}}
\newcommand{\Xcal}{\mathcal{X}}

\def\mode{\mathrel{\mkern 2mu\mathop{mod}^*}}
\renewcommand{\mod}{\mkern 2mu\mathrel{\mathop{mod}}\mkern 2mu}

\title[Restriction theory in the primes]{Notes on restriction theory in the primes}

\author{O. Ramar\'e}

\date{Received: date / Accepted: date}
\email{olivier.ramare@univ-amu.fr}
\address{CNRS/ Aix Marseille Univ, I2M, Marseille, France}
\usepackage{url}

\begin{document}

\begin{abstract}
  TO BE PUBLISHED BY ISRAEL JOURNAL OF MATHEMATICS.
We study the mean
  $\sum_{x\in\Xcal} \bigl|\sum_{p\le N}{}u_p e(xp)\bigr|^{\ell}$ when
  $\ell$ covers the full range $[2,\infty)$ and
  $\Xcal\subset\mathbb{R}/\mathbb{Z}$ is a {well-spaced} set,
  providing a smooth transition from the case $\ell=2$ to the case
  $\ell>2$ and improving on the results of J.~Bourgain and of B.~Green
  and T.~Tao. A uniform Hardy-Littlewood property for the set of
  primes
  is established as well as a sharp upper bound
  for $\sum_{x\in\Xcal} \bigl|\sum_{p\le N}{}u_p e(xp)\bigr|^{\ell}$
  when $\Xcal$ is small. These results are extended to primes in
  \emph{any} interval in a last section, provided the primes are
  numerous enough therein.
\end{abstract}
\maketitle
\tableofcontents
\section{Introduction and some results}

\subsubsection*{Some historical background}

During the proof of Theorem~3 of \cite{Bourgain*89-2} and by using
specific properties of the primes, J.~Bourgain established
 (in Equation~$(4.39)$ therein) the estimate
\begin{equation}
  \label{initJB}
  \forall \ell>2,\quad
  \biggl(\int_0^1\biggl|\sum_{p\le N} {}u_p e(p
  \alpha)\biggr|^{\ell}d\alpha\biggr)^{1/\ell}
  \ll_\ell N^{-1/\ell}\biggl(\frac{N}{\log N}
  \sum_{p\le  N}|{}u_p|^2\biggr)^{1/2}.
\end{equation}
This proof was improved by B.~Green in \cite[Theorem 1.5]{Green*05}
and in
the paper
\cite{Green-Tao*04}, B.~Green \& T.~Tao reduced the proof to using
only sieve properties, enabling a wild generalization.
A striking feature of~\eqref{initJB} is that it is \emph{not} valid
when $\ell=2$ as Parseval formula easily shows. Understanding the
transition became then an open question, an answer to which is
provided in Corollary~\ref{mainCor} below.  

Let us mention that, before the work of B.J.~Green, it was customary in
prime number theory to restrict our attention to the case $\ell=2$,
while Green used $\ell=5/2$.
This proved to be very valuable in applications.

\subsubsection*{The heart of the matter}

\begin{thm}
  \label{Extensionprecise}
  Let $\Xcal$ be a $\delta$-well spaced subset of
  $\mathbb{R}/\mathbb{Z}$ and $N\ge1000$.
  Let $({}u_p)_{p\le N}$ be a sequence of complex numbers.
  We have
  \begin{equation*}
    \sum_{x\in\Xcal}
    \biggl|\sum_{ p\le N}{}u_p e(xp)\biggr|^2
    \le
     280 \frac{N+\delta^{-1}}{\log N}\log(2|\Xcal|)\sum_{p\le N}|{}u_p|^2
    .
  \end{equation*}
\end{thm}
\noindent
Let us recall that a set $\Xcal\subset\mathbb{R}/\mathbb{Z}$ is said
to be $\delta$-well spaced when $\min_{x\neq
  x'\in\Xcal}|x-x'|_{\mathbb{Z}}\ge\delta$, where
$|y|_{\mathbb{Z}}=\min_{k\in\mathbb{Z}}|y-k|$ denotes in a rather
unusual manner the distance to
the nearest integer.
In most applications, $\delta^{-1}$ is smaller
than~$N$.

 B.J.~Green \& T.~Tao's result in~\cite{Green-Tao*04} relates to a similar inequality
though with a larger dependence in $|\Xcal|$ that the $\log(2|\Xcal|)$ we
have here. We shall prove this inequality in dual form in Theorem~\ref{Bel}.

Though Theorem~\ref{Extensionprecise} is an $L^2$-estimate, a
fundamental \emph{maximal} character is hidden in the fact that the
set $\mathcal{X}$ may be chosen freely.

\subsubsection*{Sieves and transference principle}
The main ingredient to prove Theorem~\ref{Extensionprecise} is the
large sieve inequality coupled with an \emph{enveloping sieve}; our
novelty with respect to \cite{Ramare*95} is to incorporate a
preliminary \emph{unsieving} into this sieving process. We shall spend some time
to describe properly this enveloping sieve.

In some sense, \emph{sieving}, and this is all the more true in the
context of the large sieve, relies on describing a sequence through
congruence properties. As a consequence, properties of
arithmetical progressions may well be shared by sequences properly described
by sieves. 
The terminology
\emph{transference principle} refers here to this idea. It leads in particular to
some majorant properties as shown below.

\subsubsection*{Analytical usage }
The maximal character of Theorem~\ref{Extensionprecise} may be used to
control the size of level sets, and this is the path we follow in
this subsection.

\begin{thm}
  \label{Extension}
  Let $\Xcal$ be a $\delta$-well spaced subset of
  $\mathbb{R}/\mathbb{Z}$. Assume $N\ge 1000$ and let $h>0$. We have
  \begin{equation*}
    \sum_{x\in\Xcal}
    \biggl|\sum_{p\le N}{}u_p e(xp)\biggr|^{2+h}
    \le
    7000\bigl((1+\tfrac{3}{2\log N})^h+1/h\bigr)
    \biggl(\frac{N+\delta^{-1}}{\log N}\sum_{p\le
      N}|{}u_p|^2\biggr)^{1+h/2}.
  \end{equation*}
\end{thm}
\noindent
On taking $\Xcal=\{\beta+k/N,0\le k\le N-1\}$ and integrating over
$\beta$ in $[0,1/N]$, we get the corollary we advertised above.
\begin{cor}
  \label{mainCor}
  Assume $N\ge 1000$ and let $h>0$. We have
  \begin{equation*}
    \int_0^1
    \biggl|\sum_{p\le N}{}u_p e(p\alpha)\biggr|^{2+h}d\alpha
    \le
    7000\frac{(1+\tfrac{3}{2\log N})^h+1/h}{N}
    \biggl(\frac{2N}{\log N}\sum_{p\le
      N}|{}u_p|^2\biggr)^{1+h/2}.
  \end{equation*}
\end{cor}
\noindent
This result offers an optimal (save for the constant)
transition to the case $h=0$. Indeed, on selecting $h=1/\log N$, this
corollary implies that, when $|{}u_p|\le 1$, we have the best possible
\begin{equation*}
  \int_0^1
  \biggl|\sum_{p\le N}{}u_p e(p\alpha)\biggr|^{2+h}d\alpha
  \ll
      \sum_{p\le N}|{}u_p|^2.
\end{equation*}
\subsubsection*{A majorant property}
In the same line, but maybe more strikingly, our result implies a
\emph{uniform} Hardy-Littlewood majorant property, in the sense of the
paper \cite{Green-Ruzsa*04} of B.~Green \& I.~Ruzsa. 
\begin{thm}
  \label{mainBourgain}
  Assume $N\ge 10^{6}$ and let $\ell\ge2$. We have
  \begin{equation*}
    \biggl(\int_0^1
    \biggl|\sum_{p\le N}{}u_p e(p\alpha)\biggr|^{\ell}d\alpha
    \biggr)^{1/\ell}
    \le
    10^{5}
    \biggl(\int_0^1
    \biggl|\sum_{p\le N} e(p\alpha)\biggr|^{\ell}d\alpha
    \biggr)^{1/\ell}
  \end{equation*}
  as soon as $\sum_{p\le
      N}|{}u_p|^2\le \sum_{p\le
      N}1$.
\end{thm}
\noindent
In other words, the constant $C(\ell)$ in Theorem~1.5 of
\cite{Green*05} is \emph{uniformly} bounded, and in fact
by~$10^5$. Guessing and getting the optimal constant is open,
whether under the $L^\infty$-condition $|u_p|\le 1$ or under the
$L^2$-condition we use.

\subsubsection*{Arithmetical usage }
To better compare Theorem~\ref{Extensionprecise} with earlier results
and to underline its maximal character, let us recall
Theorem~5.3 of \cite{Ramare*06}: when $Q_0\le \sqrt{N}$, we have
\begin{equation}
  \label{oldram}
  \sum_{q\le Q_0}\sum_{a\mode q}
  \biggl|\sum_{n}u_n e(na/q)\biggr|^2
  \le7\frac{N\log Q_0}{\log N}\sum_{n}|u_n|^2 
\end{equation}
valid provided $u_n$ vanishes when $n$ has a prime factor less than
$\sqrt{N}$. Here is the
estimate we can now get.
\begin{cor}
  \label{Extensionprecisebis}
  Let $N\ge1000$ and $Q_0\in[2, \sqrt{N}]$.
  Let $({}u_p)_{p\le N}$ be a sequence of complex numbers.
  We have
  \begin{equation*}
    \sum_{q\le Q_0}\sum_{a\mode q}
    \max_{|\alpha-\frac{a}{q}|\le \frac{1}{qQ_0}}
    \biggl|\sum_{ p\le N}{}u_p e(\alpha p)\biggr|^2
    \le
     1200 \frac{N\log Q_0}{\log N}\sum_{p\le N}|{}u_p|^2
    .
  \end{equation*}
\end{cor}

\subsubsection*{Extensions}
Rather than restricting their attention to prime numbers,
B.~Green \& T.~Tao considered a more general setting 
that could be encompassed in the framework of \emph{sufficiently
  sifted sequences} of \cite{Ramare-Ruzsa*01}, and the same holds for
our estimates. Indeed the methods used (essentially the large sieve
inequality and
enveloping sequences) remain general enough to warrant such an
extension.
We simply present
an obvious generalization to the case of primes in some interval in the
last section.  The somewhat reverse situation of
smooth numbers has been the subject of the work \cite{Harper*16} by
A.J.~Harper, to which we borrow an idea (see around
Equation~\eqref{defGammaxi} below).

\subsubsection*{Explicit values for the constants?}

Explicit values for the constants are provided for three reasons: it avoids us saying
that these are independant of the involved parameters; it puts forward
that our argument is elementary enough; and finally, it shows
that some work is still required to improve on them and to determine
the optimal ones. We did not work overmuch on these
constants.

\subsubsection*{Notation}
As this paper may be of interest to audiences having different
background, let us review the notation we use throughout this paper.
\begin{itemize}
\item The number of prime factors of the $\ell$ is denoted by $\omega(\ell)$.
\item The M\"obius fuction is denoted by $\mu$, so that $\mu(\ell)$
  vanishes when $\ell$ is divisible by the square of a prime, and
  otherwise takes the value $(-1)^{\omega(\ell)}$. In
  particular, we have $\mu(1)=1$.
\item The gcd of the two integers $a$ and $b$ is often denoted by $(a,b)$
  while their lcm is denoted by $[a,b]$.
\item The value of the Euler $\varphi$-function at the positive
  integer $\ell$ is the number of integers in $\{1,\cdots,\ell\}$ that
  are prime to~$\ell$. In particular $\varphi(1)=1$.
\item Notation ``$a\mode q$'' denotes a variable $a$ that ranges
  through the invertible (also called \emph{reduced}) residue classes modulo~$q$.
\item Summations are usually over positive integers when the summation
  variable is not denoted by $p$, in which case the variable $p$ runs
  through the primes. The stated conditions apply. In case more
  clarification seems necessary, we shall for instance write
  $\displaystyle\sum_{\substack{n: n\le N, \\n|d}}$ to denote a sum
  over the positive integers $n\le N$ that divide the
  parameter~$d$. 
\item The Ramanujan sum $c_q(n)$ is defined by
  \begin{equation}
    \label{defcq}
    c_q(n)=\sum_{a\mode q}e(na/q),\quad \Bigl(e(\alpha)=\exp(2i\pi \alpha)\Bigr)
  \end{equation}
  and where $a\mod^*q$ is a shortcut for ``$1\le a\le q$ and
  $(a,q)=1$''. This quantity can also be computed
  via the von Sterneck expression
  \begin{equation}
    \label{vSe}
    c_q(n)=\mu\biggl(\frac{q}{(q,n)}  \biggr)
    \frac{\varphi(q)}{\varphi\Bigl(\frac{q}{(q,n)}\Bigr)}.
  \end{equation}
\end{itemize}

\subsubsection*{Acknowledgment}
This work has been completed when the author was enjoying the hospitality of
the Hausdorff Research Institute for Mathematics in Bonn in june 2021.
\subsubsection*{Thanks}
Thanks are due to the referee for her/his careful reading and very helpful comments.

\section{Handling the $G$-functions}
\label{AMF}
We define $P(z_0)=\prod_{p<z_0}p$ and
\begin{equation}
  \label{defGd}
  G_d(y;z_0)
  =\sum_{\substack{\ell\le y,\\
      (\ell,dP(z_0))=1}}\frac{\mu^2(\ell)}{\varphi(\ell)},
  \quad
  G(y;z_0)=G_1(y;z_0).
\end{equation}
When $z_0=2$, these functions are classical in sieve theory (see for
instance Equation $(1.3)$ of \cite[Chapter 3]{Halberstam-Richert*81}
by H.~Halberstam and H.E.~Richert)and we shall in fact reduce the
analysis to this case. So we use the specific notation:
\begin{equation}
  \label{defGz}
  G(y)=G(y;2)=\sum_{\ell\le y}\frac{\mu^2(\ell)}{\varphi(\ell)}.
\end{equation}
Here are the six lemmas we will combine for their evaluations.
\begin{lem}
  \label{evalG1}
  We have $\displaystyle G(yd;z_0)\ge \frac{d}{\varphi(d)}G_d(y;z_0)\ge G(y;z_0)$
\end{lem}
\begin{lem}
  \label{evalG2}
  We have $\displaystyle \prod_{p<z_0}\frac{p}{p-1}G(z;z_0)\ge
   G(z)$.
\end{lem}
\begin{lem}
  \label{auxvarphi}
  When $h\ge0$, we have
  $\displaystyle \sum_{\ell\le y}\frac{\mu^2(\ell)}{\varphi(\ell)^{1+h}}
  \ge \sum_{q\le y}\frac{1}{q^{1+h}}$.
\end{lem}
\begin{lem}
  \label{evalG3}
  When $z\ge1$, we have $G(z)\ge \log z$.
\end{lem}

\begin{lem}
  \label{getVz0}
  When $z_0\ge2$, we have
  $\displaystyle
    \prod_{p< z_0}\frac{p-1}{p}\ge
      \frac{e^{-\gamma}}{\log(2 z_0)}$.
\end{lem}
\begin{lem}\label{EstGdown}
  When $z_0\ge2$, we have
  $\displaystyle
  G(z;z_0)\ge e^{-\gamma}\frac{\log z}{\log(2z_0)}$.
\end{lem}

\begin{proof}[Proof of Lemma~\ref{evalG1}]
  This inequality has its origin in
  \cite[Eq. (1.3)]{van-Lint-Richert*65} by J.~van Lint and
  H.E.~Richert, but the argument is so simple that we reproduce it.
  We write
  \begin{equation*}
    G(y;z_0)
    =
      \sum_{\delta|d}
      \sum_{\substack{\ell\le y,\\
      (\ell,P(z_0))=1,\\ (\ell,d)=\delta}}\frac{\mu^2(\ell)}{\varphi(\ell)}
    =
      \sum_{\delta|d}\frac{\mu^2(\delta)}{\varphi(\delta)}
      \sum_{\substack{m\le y/\delta,\\
      (m,dP(z_0))=1}}\frac{\mu^2(m)}{\varphi(m)}.
  \end{equation*}
  The last transformation is due to the fact that $\ell$ being squarefree, the
  condition $(\ell,d)=\delta$ implies that $\ell=\delta m$ with $m$
  squarefree and prime to~$d$. The final step is to notice that
  \begin{equation*}
    \sum_{\delta|d}\frac{\mu^2(\delta)}{\varphi(\delta)}=\frac{d}{\varphi(d)}
  \end{equation*}
  and that, when $\delta|d$, we have
  \begin{equation*}
    \sum_{\substack{m\le y,\\
        (m,dP(z_0))=1}}\frac{\mu^2(m)}{\varphi(m)}
    \ge
    \sum_{\substack{m\le y/\delta,\\
        (m,dP(z_0))=1}}\frac{\mu^2(m)}{\varphi(m)}
    \ge
    \sum_{\substack{m\le y/d,\\
        (m,dP(z_0))=1}}\frac{\mu^2(m)}{\varphi(m)}.
  \end{equation*}
  The lemma follows readily.
\end{proof}
\begin{proof}[Proof of Lemma~\ref{evalG2}]
  Let us change the notation in Lemma~\ref{evalG1} and use $z_1$
  rather than $z_0$.
  Then Lemma~\ref{evalG2} follows from Lemma~\ref{evalG1} with $z_1=2$
  and $d=P(z_0)$, and on recalling definition~\eqref{defGz}.
\end{proof}
\begin{proof}[Proof of Lemma~\ref{auxvarphi}]
  We simply notice that
  \begin{align*}
    \sum_{\ell\le y}\frac{\mu^2(\ell)}{\varphi(\ell)^{1+h}}
    &\ge
      \sum_{\ell\le y}\frac{\mu^2(\ell)}{\ell^{1+h}}
      \prod_{p|\ell}\biggl(\sum_{k\ge0}\frac{1}{p^k}\biggr)^{1+h}
    \\&\ge
    \sum_{\ell\le y}\frac{\mu^2(\ell)}{\ell^{1+h}}
    \prod_{p|\ell}\biggl(\sum_{k\ge0}\frac{1}{p^{k(1+h)}}\biggr)
    =
    \sum_{\substack{q\ge1,\\ k(q)\le y}}\frac{1}{q^{1+h}}\ge
    \sum_{q\le y}\frac{1}{q^{1+h}}
  \end{align*}
  where $k(d)=\prod_{p|d}p$ is the so-called \emph{squarefree kernel}
  of $d$. The lemma is proved.
\end{proof}
\begin{proof}[Proof of Lemma~\ref{evalG3}]
  This lemma is as classical as the lemmas in this section. It can for
  instance be found page~24 of the book~\cite{Bombieri*74} by E. Bombieri.  An
  upper bound of similar strength can be found in \cite[Lemma
  3.5]{Ramare*95}. For a proof, use Lemma~\ref{auxvarphi} with $h=0$ and recall
  that $\sum_{q\le y}1/q\ge \log y$.
\end{proof}
\begin{proof}[Proof of Lemma~\ref{getVz0}]
  In~\cite[Theorem 8]{Rosser-Schoenfeld*62} by
  J.B.~Rosser \& L.~Schoenfeld, we find the estimate
  $\displaystyle
    \prod_{p<z_0}\frac{p}{p-1}< e^{\gamma}\log z_0
    \biggl(1+\frac{1}{2\log^2z_0}\biggr)$ which is valid when $z_0>286$.
 Hence, when $z_0>286$ we find that
  \begin{equation*}
    e^\gamma\log (2z_0)\prod_{p< z_0}\frac{p-1}{p}\ge
    \biggl(1+\frac{\log 2}{\log z_0}\biggr)
    \biggl(1+\frac{1}{2\log^2z_0}\biggr)^{-1}\ge1
  \end{equation*}
  as $\log 2\ge 1/2$.
  A direct inspection using Pari-GP \cite{Pari-GP} establishes the stated inequality
  for the remaining values of $z_0$.
\end{proof}

\begin{proof}[Proof of Lemma~\ref{EstGdown}]
  Though the proof that follows does not require it, let us notice that the lemma is obvious when $\log z_0\ge
  e^{-\gamma}\log z$, as $G(z;z_0)\ge1$ (consider the contribution of
  the summand~$\ell=1$ in \eqref{defGd}). For the proof, simply
  combine Lemma~\ref{evalG2} together with Lemma~\ref{getVz0}.
\end{proof}
\section{Auxiliary lemmas}
\begin{lem}
  \label{appGh}
  Let $h>0$. We have
  $\displaystyle
    \sum_{d\le D}\frac{\mu^2(d)}{\varphi(d)^{1+h}}\ge
    \frac{1-D^{-h}}{h}.
  $
\end{lem}

\begin{proof}
  We first appeal to Lemma~\ref{auxvarphi} and then further simplify the lower bound as follows:
  \begin{align*}
    \sum_{q\le D}\frac{1}{q^{1+h}}
    &=\int_1^D \sum_{q\le
    t}1\frac{(1+h)dt}{t^{2+h}}+\frac{[D]}{D^{1+h}}
    \\&\ge
    \int_1^D (t-1)\frac{(1+h)dt}{t^{2+h}}+\frac{D-1}{D^{1+h}}
    =\frac{h+1}{h}\biggl(1-\frac{1}{D^h}\biggr)
    +\frac{D-1}{D^{1+h}}
    \\&\ge
    \frac{1-D^{-h}}{h}+1-\frac{1}{D^{1+h}}\ge \frac{1-D^{-h}}{h}
  \end{align*}
  as required.
\end{proof}

\begin{lem}
  \label{RS}
  We have 
  $\pi(x)=\sum_{p\le x}1\le \frac{x}{\log
    x}(1+\frac{3}{2\log x})$ and $\pi(x)\le \frac{5x}{4\log
    x}$, both valid when $x\ge 114$. Finally, $\pi(x)\ge
  x/(\log x)$ when $x\ge17$.
\end{lem}
\noindent
This can be found in~\cite[Theorem 1,
Corollary~2]{Rosser-Schoenfeld*62} by J.B.~Rosser \& L.~Schoenfeld.
\begin{lem}
  \label{BeurlingF}
  Let $M\in\mathbb{R}$, and $N$ and $\delta$ be positive real number.
  There exists a smooth function $\psi$ on $\mathbb{R}$ such that
  \begin{itemize}
  \item The function $\psi$ is non-negative.
  \item When $t\in [M,M+N]$, we have $\psi(t)\ge 1$.
  \item $\psi(0)=N+\delta^{-1}$.
  \item When $|\alpha|>\delta$, we have $\hat{\psi}(\alpha)=0$.
  \item We have $\psi(t)=\Ocal_{M,N,\delta}(1/(1+|t|^2))$.
  \end{itemize}
\end{lem}
This lemma is due to A. Selberg, see \cite[Section
20]{Selberg*91}. See also \cite{Montgomery*78} by H.L. Montgomery and \cite{Vaaler*85} by J.D. Vaaler. A similar
construction but with a stronger decreasing condition can be found in
\cite[Chapter 15]{Ramare*06}, based on the paper \cite{Holt-Vaaler*96}
by {J.J.} Holt and {J.D.} Vaaler.

\section{An enveloping sieve}

We fix two real parameters $z_0\le z$ and consider the sole case of
prime numbers. It is easy to reproduce the analysis of \cite[Section
3]{Ramare-Ruzsa*01} as far as exact formulae are concerned, but one
gets easily sidetracked towards slightly different formulae. The
reader may for instance compare \cite[Lemma 4.2]{Ramare*95} and
\cite[(4.1.14)]{Ramare-Ruzsa*01}.
Similar material is also the topic of \cite[Chapter 12]{Ramare*06}.
So we present a complete analysis in
our special case. Here is the main end-product we shall use.
\begin{thm}
  \label{Fourierbetan}
  Let $z_0\le z$ be two parameters. There exists an upper bound
  $\beta_{z_0,z}$ of the characteristic function of those integers
  that do not have any prime factor in the interval $[z_0,z)$.
  The function $\beta_{z_0,z}$ admits the expansion:
  \begin{equation*}
    \beta_{z_0,z}(n)
    =\sum_{\substack{q\le z^2,\\ q| P(z)/P(z_0)}}w_q(z;z_0)c_q(n)
  \end{equation*}
  where $c_q(n)$ is the Ramanujan sum and where
  \begin{equation*}
    w_q(z;z_0)
    =
    \frac{\mu(q)}{\varphi(q)G(z;z_0)}\frac{G_{[q]}(z;z_0)}{G(z;z_0)},
  \end{equation*}
  with the definition
  \begin{equation}
    \label{defGbracketq}
    G_{[q]}(z;z_0)
    =\sum_{\substack{\ell\le z/\sqrt{q},\\ (\ell,qP(z_0))=1}}
    \frac{\mu^2(\ell)}{\varphi(\ell)}
    \xi_q(z/\ell)
  \end{equation}
  and
  \begin{equation*}
    \xi_q(y)=\sum_{\substack{q_1q_2q_3=q,\\ q_1q_3\le y, \\q_2q_3\le y}}
    \frac{\mu(q_3)\varphi_2(q_3)}{\varphi(q_3)}
    \quad\text{where}
    \quad\varphi_2(q_3)=\prod_{p|q}(p-2).
  \end{equation*}
  Notice that $\xi_q(y)=q/\varphi(q)$ when $y\ge q$ and that
  $|\xi_q(y)|\le 3^{\omega(q)}$ always.
\end{thm}
\begin{remark}
  The factor $G_{[q]}(z;z_0)/G(z;z_0)$ should be looked upon as a mild
  perturbation. It can be shown to equivalent to 1 when $q$ goes to
  infinity, and, in general, it only introduces technicalities that
  can be handled.
\end{remark}
\begin{remark}
  The coefficient $w_q(z;z_0)$ is the main actor here. It saves the
  density $1/G(z;z_0)$ of the sequence $(\beta_{z_0,z}(n))$. The
  further saving introduced by the factor $1/\varphi(q)$ is essential, though a
  milder decreasing rate is enough (see Lemma~\ref{Courageous}). It comes from an
  equidistribution of the sequence $(\beta_{z_0,z}(n))$ in invertible
  arithmetical progressions modulo~$q$. It can easily be shown that
  \begin{equation*}
    w_q(z;z_0)=\lim_{N\rightarrow\infty}\frac{1}{N}\sum_{n\le N}\beta_{z_0,z}(n)e(na/q)
  \end{equation*}
  for every $a$ prime to $q$.
\end{remark}
\begin{remark}
  The parameter $z_0$ will be essential: the coefficients $w_q(z;z_0)$
  for $q\in(1,z_0)$ vanish. See around~\eqref{MainStep}.
\end{remark}


\begin{proof}[Proof of Theorem~\ref{Fourierbetan}]
We split the proof in three steps. We follow closely Section~3 of
\cite{Ramare-Ruzsa*01}. See also \cite[Chapter 11]{Ramare*06}.
\subsubsection*{Building the upper bound}
The initial idea of the Selberg sieve is to consider the family
\begin{equation}
  \label{defbeta}
  \beta_{z_0,z}(n)=\biggl(\sum_{\substack{d:d|n,\\ (d,P(z_0))=1,\\ d\le z}}\lambda_d\biggr)^2
\end{equation}
for arbitrary real coefficients $\lambda_d$ that are only constrained
by the condition $\lambda_1=1$. Indeed, for any such set of of
coefficients, the resulting function is non-negative and takes the
value~1 at integers $n$ that have no prime factors dividing $P(z)/P(z_0)$.
After an optimization step that we skip, one reaches the choice
\begin{equation}
  \label{deflambdad}
  \lambda_d=
  \1_{(d,P(z_0))=1}\frac{\mu(d)dG_d(z/d;z_0)}{\varphi(d)G(z;z_0)}
\end{equation}
where $G_d$ is given by~\eqref{defGd} (notice that, indeed,
$\lambda_1=1$). From now onward, we reserve the notation $\lambda_d$
for this special choice. Though we shall not use it, notice that
Lemma~\ref{evalG1} implies the bound $|\lambda_d|\le 1$.

We develop the square above and get
  \begin{align*}
    \beta_{z_0,z}(n)
    &=\sum_{\substack{d_1,d_2,\\ [d_1,d_2]|n}}\lambda_{d_1}\lambda_{d_2}
      =\sum_{d_1,d_2}\frac{\lambda_{d_1}\lambda_{d_2}}{[d_1,d_2]}
      \sum_{q|[d_1,d_2]}\sum_{a\mode q}e(na/q)
    \\&=\sum_{\substack{q\le z^2,\\ (q,P(z_0))=1}}w_q(z;z_0)c_q(n)
  \end{align*}
  where
  \begin{equation}
    w_q(z;z_0)
    =
    \sum_{q|[d_1,d_2]}\frac{\lambda_{d_1}\lambda_{d_2}}{[d_1,d_2]}.
  \end{equation}
  Note that $w_q(z;z_0)=0$ when $q$ does not divide $P(z)/P(z_0)$, and
  in particular when it is not squarefree. Let us assume now that
  $q|P(z)/P(z_0)$.
\subsubsection*{Expliciting $w_q(z;z_0)$}
  We introduce the definition \eqref{deflambdad} of the $\lambda_d$'s
  and obtain
  \begin{equation*}
    G(z;z_0)^2
    w_q(z;z_0)
    =
    \sum_{\substack{\ell_1,\ell_2\le z,\\ (\ell_1\ell_2,P(z_0))=1}}
    \frac{\mu^2(\ell_1)}{\varphi(\ell_1)}
    \frac{\mu^2(\ell_2)}{\varphi(\ell_2)}
    \sum_{\substack{q|[d_1,d_2],\\ d_1|\ell_1, d_2|\ell_2}}
    \frac{d_1\mu(d_1)d_2\mu(d_2)}{[d_1,d_2]}.
  \end{equation*}
  The inner sum vanishes is $\ell_1$ has a prime factor prime to
  $q\ell_2$, and similarly for $\ell_2$. Furthermore, we need to have
  $q|[\ell_1,\ell_2]$ for the inner sum not to be empty. Whence we may
  write $\ell_1=q_1q_3\ell$ and $\ell_2=q_2q_3\ell$ where $(\ell,q)=1$
  and $q=q_1q_2q_3$. The part of the inner sum corresponding to $\ell$
  has value $\prod_{p|\ell}(p-2+1)=\varphi(\ell)$. We have reached
  \begin{equation*}
    G(z;z_0)^2
    w_q(z;z_0)
    =
    \mkern-10mu
    \sum_{\substack{\ell\le z,\\ (\ell,qP(z_0))=1}}
    \mkern-5mu
    \frac{\mu^2(\ell)}{\varphi(\ell)}
    \sum_{\substack{q_1q_2q_3=q,\\ q_1q_3\ell\le z, \\q_2q_3\ell\le z}}
    \frac{1}{\varphi(q)\varphi(q_3)}
    \sum_{\substack{q|[d_1,d_2],\\ d_1|q_1q_3,\\ d_2|q_2q_3}}
    \frac{d_1\mu(d_1)d_2\mu(d_2)}{[d_1,d_2]}.
  \end{equation*}
  In this last inner sum, we have necessarily $d_1=q_1d'_1$ and
  $d_2=q_2d'_2$, so $q_3=[d'_1,d'_2]$. We may thus write
  \begin{equation*}
    G(z;z_0)^2
    w_q(z;z_0)
    =
    \mkern-10mu
    \sum_{\substack{\ell\le z,\\ (\ell,qP(z_0))=1}}
    \mkern-5mu
    \frac{\mu^2(\ell)}{\varphi(\ell)}
    \sum_{\substack{q_1q_2q_3=q,\\ q_1q_3\ell\le z, \\q_2q_3\ell\le z}}
    \frac{\mu(q)\mu(q_3)}{\varphi(q)\varphi(q_3)}
    \sum_{\substack{d'_1,d'_2:\\q_3=[d'_1,d'_2]}}
    \frac{d'_1\mu(d'_1)d'_2\mu(d'_2)}{[d'_1,d'_2]}.
  \end{equation*}
  This last inner sum has value $\varphi_2(q_3)$, whence
  \begin{equation*}
    G(z;z_0)^2
    w_q(z;z_0)
    =
    \frac{\mu(q)}{\varphi(q)}
    \sum_{\substack{\ell\le z,\\ (\ell,qP(z_0))=1}}
    \frac{\mu^2(\ell)}{\varphi(\ell)}
    \sum_{\substack{q_1q_2q_3=q,\\ q_1q_3\ell\le z, \\q_2q_3\ell\le z}}
    \frac{\mu(q_3)\varphi_2(q_3)}{\varphi(q_3)}
  \end{equation*}
  as announced. The size conditions are readily seen to imply that
  $\ell\le z/\sqrt{q}$.

\end{proof}

\begin{lem}
  \label{Courageous}
  When $4\le z_0\le z$, we have
  $\displaystyle
    |w_q(z;z_0)|
    \le 6\,\frac{\log z_0}{\sqrt{q}\log z}$.
\end{lem}

\begin{proof}
  We deduce from the definition the estimate
  $|\xi_q(y)|\le 3^{\omega(q)}$, and thus
  \begin{equation}
    \label{eq:7}
    |G(z;z_0)w_q(z;z_0)|\le 3^{\omega(q)}/\varphi(q).
  \end{equation}
  As $z_0>3$, we may assume that $q$ is prime to $6$, since otherwise $w_q(z;z_0)=0$.
  We use Lemma~\ref{EstGdown} to get
  \begin{align*}
    |w_q(z;z_0)|
    &\le \prod_{p\ge 5}
    \max\biggl(\frac{3\sqrt{p}}{p-1},1\biggr) \frac{1}{G(z;z_0)\sqrt{q}}
    \le \frac{2.23\times e^\gamma\, \log 2z_0}{\sqrt{q}\log z}
    \\&\le \frac{2.23\times e^\gamma\, \log z_0}{\sqrt{q}\log z}
    \biggl(1+\frac{\log 2}{\log 4}\biggr)
    \le 6\,\frac{\log z_0}{\sqrt{q}\log z}
\end{align*}
It has been enough, in the Euler product, to consider the primes
$p=5$ and~$p=7$. The lemma follows swiftly.
\end{proof}

\section{The fundamental estimate}

\begin{thm}
  \label{Bel} Let $N\ge 1000$.
  Let $\mathcal{B}$ be a $\delta$-well spaced subset of
  $\mathbb{R}/\mathbb{Z}$. For any function $f$
  on $\mathcal{B}$, we have
  \begin{equation*}
    \sum_{p\le N}\biggl|\sum_{b\in \mathcal{B}}f(b)e(bp)\biggr|^2
    \le 280
    (N+\delta^{-1})\|f\|_2^2\frac{\log
      (2\|f\|_1^2/\|f\|_2^2)}{\log N}.
  \end{equation*}
  where $\|f\|_q^q=\sum_{b\in \mathcal{B}}|f(b)|^q$ for any positive $q$.
\end{thm}

\begin{proof}
  Let us first notice that $\|f\|_1^2\ge\|f\|_2^2$.
  Let $z=N^{1/4}$ and
  \begin{equation*}
    z_0=\biggl(2\frac{\|f\|_1^2}{\|f\|_2^2}\biggr)^2\ge 4.
  \end{equation*}
  We have $z_0\le z$ when $\|f\|_1^2/\|f\|_2^2\le N^{1/8}/2$.
  When this condition is not met, we use the dual of the usual large sieve
  inequality (see \cite{Montgomery*78} by H.L. Montgomery) to infer that
  \begin{align*}
    \sum_{p\le N}\biggl|\sum_{b\in \mathcal{B}}f(b)e(bp)\biggr|^2
    &\le
      (N+\delta^{-1})\|f\|_2^2
    \\&
    \le  (N+\delta^{-1})\|f\|_2^2\frac{\log
    (2\|f\|_1^2/\|f\|_2^2)}{\log (N^{1/8})}.
  \end{align*}
  This establishes our inequality
  in this case.
  Henceforth, we assume that $z_0\le z$. We discard the small primes trivially:
  \begin{align*}
    \sum_{ p\le z}\biggl|\sum_{b\in \mathcal{B}}f(b)e(bp)\biggr|^2
    &\le z\|f\|_1^2\le N^{3/8}\|f\|_2^2/\sqrt{8}
    \\&\le N \frac{\|f\|_2^2\log(2{\|f\|_1^2}/{\|f\|_2^2})}{\log
    N}
    \frac{\log N}{\sqrt{8}N^{5/8}\log 2}
    \\&\le \frac{N}{2880} \frac{\|f\|_2^2\log(2{\|f\|_1^2}/{\|f\|_2^2})}{\log
    N}.
  \end{align*}
  Let us now define
  \begin{equation}
    \label{defW}
    W=\sum_{z<p\le N}\biggl|\sum_{b\in \mathcal{B}}f(b)e(bp)\biggr|^2.
  \end{equation}
   We
   bound above the characteristic function of those primes by our
  enveloping sieve and further
   majorize the characteristic function of the interval $[1,N]$ by
  a function $\psi$ (see Lemma~\ref{BeurlingF}) of Fourier transform supported by
  $[-\delta_1,\delta_1]$ where $\delta_1=\min(\delta,1/(2z^4))$, and
  which is such that $\hat{\psi}(0)=N+\delta_1^{-1}$. This
  leads to
  \begin{equation*}
    W\le
    \sum_{\substack{q\le z^2,\\
        (q,P(z_0))=1}}w_q(z;z_0)
    \sum_{a\mode  q}\sum_{b_1,b_2}f(b_1)\overline{f(b_2)}
    \sum_{n\in\mathbb{Z}}e((b_1-b_2)n)e(an/q)\psi(n).
  \end{equation*}
  We split this quantity according to whether $q< z_0$ or not:
  \begin{equation*}
    W=W(q< z_0)+ W(q\ge z_0).
  \end{equation*}
  When $q\ge z_0$,
  Poisson summation formula tells us that the inner sum
  is also $\sum_{m\in\mathbb{Z}}\hat{\psi}(b_1-b_2-(a/q)+m)$.
  The sum over $b_1$, $b_2$ and $n$ is thus
  \begin{equation*}
    \le (N+\delta^{-1}_1)\sum_{b_1,b_2}f(b_1)\overline{f(b_2)}\#\bigl\{(a/q)/
     \|b_1-b_2+a/q\|< \delta_1
    \bigr\}.
  \end{equation*}
  Given $(b_1,b_2)$, at most one $a/q$ may work, since
  $1/z^4>2\delta_1$. By bounding above $w_q(z;z_0)$ by
  Lemma~\ref{Courageous}, we see that
  \begin{align}
    W(q\ge z_0)
    &\le\notag
    6(N+\delta^{-1}_1)\frac{\|f\|_1^2\log z_0}{\sqrt{z_0}\log z}
    \\&\le\label{MainStep}
    \frac{6}{\sqrt{2}}(N+\delta^{-1}_1)\frac{\|f\|_2^2\log z_0}{\log z}.
  \end{align}
  When $w_q(z;z_0)\neq0$, we have $q|P(z)/P(z_0)$; on adding the
  condition~$q<z_0$, only $q=1$ remains.
  Since $\mathcal{B}$ is $\delta$-well-spaced and $w_1(z;z_0)=1/G(z;z_0)$,
  Lemma~\ref{EstGdown} leads to 
  \begin{equation*}
   W(q<z_0)
    \le
    (N+\delta_1^{-1})\frac{e^\gamma\|f\|_2^2\log 2z_0}{\log z}.
  \end{equation*}
  We check that $(N+\delta_1^{-1})\le
  \frac{N+4N}{N}(N+\delta^{-1})$. We finally get
  \begin{multline*}
    \sum_{ p\le N}\biggl|\sum_{b\in \mathcal{B}}f(b)e(bn)\biggr|^2\le
    \biggl(\frac1{2880}
    +5\times 2\times 4\times
    \biggl(\frac{6}{\sqrt{2}}+e^\gamma\biggl(1+\frac{\log 2}{\log z_0}\biggr)\biggr) \biggr)
    \\\times(N+\delta^{-1})\|f\|_2^2\frac{\log
    (2\|f\|_1^2/\|f\|_2^2)}{\log N}.
  \end{multline*}
  The proof of the theorem follows readily.
\end{proof}

\section{On moments. Proof of Theorem~\ref{Extension}}

\begin{lem}
  \label{easy}
  Assume $y/\log y \le t$ with $y\ge2$ and $t\ge e$. Then $y\le 2t\log t$.
\end{lem}

\begin{proof}
  Our property is trivial when $y\le 2e$.
  Notice that the function $f:y\mapsto y/\log y$ is non-increasing when $y\ge
  e$. We find that $f(2t\log t)\ge t\ge f(y)$, whence $2t\log t\ge y$
  as sought.
\end{proof}

\begin{proof}[Proof of Theorem~\ref{Extension}]
  For typographical simplification, we define
  \begin{equation}
    \label{defB}
    B = \biggl(\frac{N+\delta^{-1}}{\log N}\sum_{p\le
      N}|{}u_p|^2\biggr)^{1/2}.
  \end{equation}
  We also set $\ell=2+h$.
  For any $\xi>0$, we examine  the set
  \begin{equation}
    \label{defXclaxi}
    \Xcal_\xi=\biggl\{x\in\Xcal/\biggl|\sum_{ p\le
      N}{}u_p e(xp)\biggr|
    \ge \xi B\biggr\}.
  \end{equation}
  By the Cauchy-Schwartz inequality and Lemma~\ref{RS}, we see that $\xi \le
  c_1=\min(5/4,1+\tfrac{3}{2\log N})$ or else, the set
  $\Xcal_\xi$ is empty.
  We consider
  (as in \cite{Harper*16}, bottom of page~1141, by A.J. Harper)
  \begin{equation}
    \label{defGammaxi}
    \Gamma(\xi)=\sum_{x\in\Xcal_\xi}
    \biggl|\sum_{p\le N}{}u_p e(xp)\biggr|
    .
  \end{equation}
  We write it as
  $
    \Gamma(\xi)
    =\sum_{x\in\Xcal_\xi}
    c(x)\sum_{p\le N}{}u_pe(xp)
  $
  for some $c(x)$ of modulus~1 and develop it in
  \begin{equation*}
    \Gamma(\xi)
    =
    \sum_{p\le N}{}u_p\sum_{x\in\Xcal_\xi}c(x)e(xp).
  \end{equation*}
  We apply Cauchy's inequality to this expression to get
  \begin{equation*}
    \Gamma(\xi)^2
    \le \sum_{p\le N}|{}u_p|^2
    \sum_{p\le N}\biggl|\sum_{x\in\Xcal_\xi}c(x)e(xp)\biggr|^2
    \le 280\, B^2|\Xcal_\xi|\log(2|\Xcal_\xi|)
  \end{equation*}
  by Theorem~\ref{Bel}.
  It follows from this upper bound that
  \begin{equation*}
    \xi^2|\Xcal_\xi|^2B^2
    \le \Gamma(\xi)^2
    \le 280\, 
   B^2|\Xcal_\xi|\log(2|\Xcal_\xi|)
  \end{equation*}
  whence
  \begin{equation*}
    2|\Xcal_\xi|/\log(2|\Xcal_\xi|)\le 560\, /\xi^2.
  \end{equation*}
  We convert this inequality via Lemma~\ref{easy} in
  $2|\Xcal_\xi|\le 1120\xi^{-2}\log(560/\xi^2)$.

  We can now turn towards the proof of the stated inequality and
  select $\xi_j=c_1/c^j$ for some $c>1$ that we will select later. We get
  \begin{align*}
    \sum_{x\in\Xcal}
    \biggl|\sum_{p\le N}{}u_p e(xp)\biggr|^\ell/B^\ell
    &\le
      c_1^\ell |\Xcal_{\xi_0}|
      +
      \sum_{j\ge1}
      \frac{c_1^\ell}{c^{\ell j}}(|\Xcal_{\xi_j}|-|\Xcal_{\xi_{j-1}}|)
    \\&\le 560(1-c^{-\ell})
    \sum_{j\ge0}
      \frac{c_1^{\ell-2}(\log(560)-2\log c_1+2j\log c)}{c^{(\ell-2) j}}.
    \\&\le 560
    \sum_{j\ge0}
      \frac{c_1^{\ell-2}(7+2j\log c)}{c^{(\ell-2) j}}.
  \end{align*}
  We check that
  \begin{equation*}
    560
    \sum_{j\ge0}
    \frac{c_1^{\ell-2}\times 7}{c^{(\ell-2) j}}
    =\frac{560\times 7\times c_1^{\ell-2}}{1-c^{2-\ell}}
  \end{equation*}
  and that
  \begin{align*}
    560
    \sum_{j\ge1}
    \frac{c_1^{\ell-2} j\,2\log c}{c^{(\ell-2) j}}
    &\le
    1120 \frac{(\log c)}{c^{\ell-2}}c_1^{\ell-2}
    \sum_{j\ge1}
    \frac{j}{c^{(\ell-2) (j-1)}}
    \\&\le \frac{1120\times (c_1/c)^{\ell-2}\log c}{(1-c^{2-\ell})^2}.
  \end{align*}
  When $\ell\ge3$, we select $c=2$, getting after some numerical work
  \begin{equation*}
    \sum_{x\in\Xcal}
    \biggl|\sum_{p\le N}{}u_p e(xp)\biggr|^{2+h}
    \le
    (3920(1+\tfrac{3}{2\log N})^h+2000)
    \biggl(\frac{N+\delta^{-1}}{\log N}\sum_{p\le
      N}|{}u_p|^2\biggr)^{1+h/2}.
  \end{equation*}
  When $\ell\in(2,3)$, we select $c=\exp(1/h)$, getting similarly
  \begin{equation*}
    \sum_{x\in\Xcal}
    \biggl|\sum_{p\le N}{}u_p e(xp)\biggr|^{2+h}
    \le
    \biggl( 6300(1+\tfrac{3}{2\log N})^h
    +\frac{2900}{h}\biggr)
    \biggl(\frac{N+\delta^{-1}}{\log N}\sum_{p\le
      N}|{}u_p|^2\biggr)^{1+h/2}.
  \end{equation*}
  Our theorem follows readily.
\end{proof}

\section{Large sieve bound on small sets. Proof of Theorem~\ref{Extensionprecise}}

\begin{proof}[Proof of Theorem~\ref{Extensionprecise}]
  This is a classical argument of duality. We write
  \begin{equation*}
    V=\sum_{x\in\Xcal}
    \biggl|\sum_{p\le N}{}u_p e(xp)\biggr|^2
    =\sum_{x\in\Xcal}\sum_{p\le N}{}u_p\overline{S(x)}e(xp)
  \end{equation*}
  where $S(x)=\sum_{1\le p\le N}{}u_p e(xp)$. On using the
  Cauchy-Schwarz inequality, we get
  \begin{equation*}
    V^2\le \sum_{p\le N}|{}u_p|^2
   \sum_{p\le N} \Bigl|\sum_{x\in\Xcal}\overline{S(x)}e(xp)\Bigr|^2.
 \end{equation*}
 We invoke Theorem~\ref{Bel} and notice to control $\|S\|_1^2/\|S\|_2^2$ that
 \begin{equation*}
   \Bigl(\sum_{x\in\Xcal}|\overline{S(x)}|\Bigr)^2
   \le |\Xcal|\sum_{x\in\Xcal}|\overline{S(x)}|^2.
 \end{equation*}
 This leads to
 \begin{equation*}
   V^2\le 280\, \frac{N+\delta^{-1}}{\log N}\sum_{p\le N}|{}u_p|^2
   \sum_{x\in\Xcal}|S(x)|^2\log(2|\Xcal|).
 \end{equation*}
 On simplifying by $\sum_{x\in\Xcal}|S(x)|^2$ (after discussing
 whether it vanishes or not), we get our estimate.
\end{proof}

\section{Optimality and uniform boundedness. Proof of Theorem~\ref{mainBourgain}}

\begin{proof}[Proof of Theorem~\ref{mainBourgain}]

We assume that $N\ge 10^6$ and set
\begin{equation}
  \label{eq:1}
  S(\alpha)=\sum_{p\le N}e(p\alpha).
\end{equation}
The argument employed at the bottom of page 1626 of \cite{Green*05}
by B.~Green
is not enough for us. Instead, we got our inspiration from the argument developped by
R.C. Vaughan in~\cite{Vaughan*88}. It runs as follows. We first notice that
\begin{equation*}
  \biggl|\sum_{a\mode q}S\Bigl(\frac aq+\beta\Bigr)\biggr|
  \le
  \biggl(\sum_{a\mode q}
  \Bigl|S\Bigl(\frac aq+\beta\Bigr)\Bigr|^{\ell}\biggr)^{1/\ell}
  \biggl(\sum_{a\mode q}1\biggr)^{(\ell-1)/\ell}.
\end{equation*}
A direct inspection shows that
\begin{equation*}
  \sum_{a\mode q}S\Bigl(\frac aq+\beta\Bigr)
  =\mu(q) S(\beta)+T(q,\beta)
\end{equation*}
where
\begin{equation}
  \label{defTqbeta}
  T(q,\beta)= \sum_{\substack{p:p|q}}e(p\beta)(c_q(p)-\mu(q)).
\end{equation}
 The bound
$|c_q(n)|\le \varphi((n,q))$ for the Ramanujan sum $c_q(n)$ (use for
instance the von Sterneck expression~\eqref{vSe}) gives us
\begin{equation}
  \label{eq:10}
  |T(q,\beta)|\le \sum_{\substack{p:p|q}}(p-1+1)\le q.
\end{equation}
The last inequality follows from the trivial property that a sum of
positive integers is certainly not more than its product.
We next get a lower bound for $S(\beta)$ by writing
\begin{equation*}
  1-e(\beta p)=2i\pi\int_0^{\beta p}e(t)dt
\end{equation*}
whence
\begin{equation}
  \label{you}
  |S(\beta)|\ge S(0)-2\pi \beta N S(0)\ge (1-2\pi\beta N) S(0)
  \ge (1-2\pi\beta N)\frac{N}{\log N}
\end{equation}
by Lemma~\ref{RS}. When $|\beta|\le 1/(7N)$, this leads to
$|S(\beta)|\ge c_2 N/\log N$ with $c_2=1-2\pi/7$, and, when $q$ is
squarefree and not more than $\sqrt{N}$, to
\begin{equation}
  \label{eq:13}
  \bigl|\mu(q) S(\beta)+T(q,\beta)\bigr|
  \ge
  c_2\frac{N}{\log N}-\sqrt{N}\ge \frac{N}{12\log N}.
\end{equation}
We thus get, when $|\beta|\le 1/(7N)$,
\begin{equation*}
  \sum_{a\mode q}
  \Bigl|S\Bigl(\frac aq+\beta\Bigr)\Bigr|^{\ell}
  \ge
  \frac{\mu^2(q)}{\varphi(q)^{\ell-1}}
  |S(\beta)+T(q,\beta)|^\ell\ge
  \frac{\mu^2(q)}{\varphi(q)^{\ell-1}}
  \biggl(\frac{N}{12\log N}\biggr)^\ell.
\end{equation*}
Thus
\begin{align*}
  \int_0^1 |S(\alpha)|^\ell d\alpha
  &\ge
  \sum_{q\le \sqrt{N}} \sum_{a\mode q}\mu^2(q)
  \int_{\frac aq-\frac{1}{7N}}^{\frac{a}{q}+\frac{1}{7N}}
  \Bigl|S\Bigl(\frac aq+\beta\Bigr)\Bigr|^{\ell}d\beta
  \\&\ge
  \frac{2}{7N}\sum_{q\le \sqrt{N}} \frac{\mu^2(q)}{\varphi(q)^{\ell-1}}
  \biggl(\frac{N}{12\log N}\biggr)^\ell.
\end{align*}
By Lemma~\ref{appGh}, we conclude
that
\begin{equation*}
  \int_0^1 |S(\alpha)|^\ell d\alpha
  \ge
  \frac{1-\sqrt{N}^{2-\ell}}{\ell-2}\frac{2}{7N}
  \biggl(\frac{N}{12\log N}\biggr)^\ell.
\end{equation*}
We thus find that
\begin{equation}
  \label{refK}
    \int_0^1
    \biggl|\sum_{p\le N}{}u_p e(p\alpha)\biggr|^{\ell}d\alpha
    \le
    K(\ell)\biggl(\frac{\log N}{N}\sum_{p\le
      N}|{}u_p|^2\biggr)^{\ell/2}
    \int_0^1  \biggl|\sum_{p\le N} e(p\alpha)\biggr|^{\ell} d\alpha
\end{equation}
where $\ell=2+h$ and
\begin{align*}
  K(2+h)
  &=
  7000\frac{(1+\tfrac{3}{2\log N})^h+1/h}{N}
  \biggl(\frac{2N^2}{(\log N)^2}\biggr)^{\ell/2}
   \frac{h}{1-\sqrt{N}^{-h}}\frac{7N}{2}
  \biggl(\frac{N}{12\log N}\biggr)^{-\ell}
  \\&=
  7000\biggr((1+\tfrac{3}{2\log N})^h+1/h\biggl)
  2^{\ell/2}
   \frac{h}{1-\sqrt{N}^{-h}}\frac{7}{2}
  12^{\ell}.
\end{align*}
When $h\ge1$,  we use
\begin{equation*}
  K(2+h)
  \le 24500\bigr(1.11^hh+1\bigl)
   \frac{1}{0.999}
   (12\sqrt{2})^{\ell}
   \le 10^7\cdot 20^\ell
 \end{equation*}
 where the worst case is reached next to $\ell=18.19\cdots$.
 When $h<1$, the quantity $K(2+h)$
is bounded above by
\begin{equation*}
  \frac{3\cdot 10^{8}}{h}\frac{h}{1-\sqrt{N}^{-h}}=
  \frac{3\cdot 10^{8}}{1-\sqrt{N}^{-h}}.
\end{equation*}
This is bounded above by $8\cdot 10^{8}$ when $h\ge 1/\log N$.
When $0\le h\le 1/\log N$, we use
\begin{align*}
  \int_0^1
  \biggl|\sum_{p\le N}{}u_p e(p\alpha)\biggr|^{\ell}d\alpha
  &\le \left(\pi(N)\sum_{p\le N}|u_p|^2\right)^{h/2}\int_0^1
  \biggl|\sum_{p\le N}{}u_p e(p\alpha)\biggr|^2 d\alpha
  \\&\le 
  \sqrt{5/4}\biggl(\frac{\log N}{N}\sum_{p\le
  N}|{}u_p|^2\biggr)^{\ell/2}
  \biggl(\frac{N}{\log N}\biggr)^{\ell}N^{-1}\log N
\end{align*}
which leads to~\eqref{refK} with
\begin{align*}
  K(2+h)
  &=\frac{\sqrt{5/4}}{N}\biggl(\frac{N}{\log N}\biggr)^{\ell}
  \frac{h\log N}{1-\sqrt{N}^{-h}}\frac{7N}{2}
    \biggl(\frac{N}{12\log N}\biggr)^{-\ell}
  \\&\le 
  \frac{7\sqrt{5/4}\cdot 12^3}{2(1-\exp(-1/2))}\le 8\cdot 10^8.
\end{align*}
Theorem~\ref{mainBourgain} follows readily.
\end{proof}

\section{Small sets large sieve estimates. Proof  of Corollary~\ref{Extensionprecisebis}}

\begin{proof}[Proof of Corollary~\ref{Extensionprecisebis}]
  The split the Farey sequence
  \begin{align}
    \label{eq:2}
    F(Q_0)
    &=\Bigl\{\frac{a}{q}, 1\le a\le q\le Q_0,(a,q)=1\Bigr\}
    \\&=\big\{0<x_1<x_2<\ldots<x_K=1\bigr\}\notag
  \end{align}
  in $F_1(Q_0)=\{x_{2i}, 1\le i\le K/2\}$ union
  $F_2(Q_0)=\{x_{2i+1}, 1\le 0\le (K-1)/2\}$. We recall that the
  distance between two consecutive points $a/q$ and $a'/q'$ in
  $F(Q_0)$ is $1/(qq')$; this is at least as large as $
  \frac{1}{qQ_0}+\frac{1}{q'Q_0}$ by the known property $q+q'\ge
  Q_0$. Hence two intervals
  $[\frac{a_1}{q_1}-\frac{1}{q_1Q_0},\frac{a_1}{q_1}+\frac{1}{q_1Q_0}]$
  and
  $[\frac{a_2}{q_2}-\frac{1}{q_2Q_0},\frac{a_2}{q_2}+\frac{1}{q_2Q_0}]$
  with $\frac{a_1}{q_1}, \frac{a_2}{q_2}\in F_1(Q_0)$ are separated by
  at least $1/Q_0^2$. We check this is also true when seen on the unit
  circle: the largest point of $F(Q_0)$ is $1$ and its smallest is
  $\frac{1}{[Q_0]}$. The same applies to $F_2(Q_0)$.
  We finally notice that $|F(Q_0)|\le Q_0(Q_0+1)/2\le Q_0^2$.

  To prove our corollary, for every $x_{2i}\in F_1(Q_0)$, we select a
  point $\tilde{x}_{2i}$ such that
  \begin{equation}
    \label{eq:9}
    \Bigl|\sum_{p\le N}u_p e(p\tilde{x_{2i}})\Bigr|
    =\max_{|x-x_{2i}|\le \frac{1}{qQ_0}}\Bigl|\sum_{p\le N}u_p e(p{x})\Bigr|
  \end{equation}
  and apply Theorem~\ref{Extensionprecise} to the set
  $\tilde{X}_1=\{\tilde{x}_{2i}\}$. We proceed similarly with $F_2(Q_0)$.
  The last details are left to
  the readers.
\end{proof}

\section{Extension to primes in intervals}

We discuss here how our results extend from the case of primes in
the initial interval to primes in $[M+1, M+N]$ for some non-negative~$M$. During
the proof of Theorem~\ref{Bel}, we used
the property that our sequence has at most $N^{1/4}$ elements below
$N^{1/4}$, and that the remaining ones are prime to any integer below $N^{1/4}$.
This is certainly still true when looking at intervals.

\begin{thm}
  \label{Belplus} Let $N\ge 1000$.
  Let $B$ be a $\delta$-well spaced subset of
  $\mathbb{R}/\mathbb{Z}$. For any function $f$
  on $\mathcal{B}$, we have
  \begin{equation*}
    \sum_{M+1\le p\le M+N}\biggl|\sum_{b\in \mathcal{B}}f(b)e(bp)\biggr|^2
    \le 280
    (N+\delta^{-1})\|f\|_2^2\frac{\log
    (2\|f\|_1^2/\|f\|_2^2)}{\log N}.
  \end{equation*}
\end{thm}
\noindent
When defining
$c_1$ in the proof of Theorem~\ref{Extension}, we used an upper bound
for the number of elements in our set. The 
version of the Brun-Titchmarsh inequality proved by H.~Montgomery \&
R.C.~Vaughan in~\cite{Montgomery-Vaughan*74} enables us to use
$c_1=2$.
After some trivial modifications, we reach the following.
\begin{thm}
  \label{Extensionplus}
  Let $\Xcal$ be a $\delta$-well spaced subset of
  $\mathbb{R}/\mathbb{Z}$. Assume $N\ge 1000$ and let $h>0$. We have
  \begin{equation*}
    \sum_{x\in\Xcal}
    \biggl|\sum_{M+1\le p\le M+N}\mkern-16mu{}u_p e(xp)\biggr|^{2+h}
    \mkern-14mu\le
    14000\bigl(2^h+1/h\bigr)
    \biggl(\frac{N+\delta^{-1}}{\log N}\sum_{M+1\le p\le
      M+N}\mkern-16mu|{}u_p|^2\biggr)^{1+h/2}.
  \end{equation*}
\end{thm}
\begin{cor}
  \label{mainCorbis}
  Assume $N\ge 1000$ and let $h>0$. We have
  \begin{equation*}
    \int_0^1
    \biggl|\sum_{M+1\le p\le M+N}\mkern-16mu{}u_p e(p\alpha)\biggr|^{2+h}
    \mkern-11mu
    d\alpha\le
    14000\frac{2^h+1/h}{N}
    \biggl(\frac{2N}{\log N}\sum_{M+1\le p\le
      M+N}\mkern-16mu|{}u_p|^2\biggr)^{1+h/2}.
  \end{equation*}
\end{cor}
\noindent
Modifying the proof of Theorem~\ref{mainBourgain} is more delicate as
it requires bounding the trigonometric polynomial $S$ from below
in~\eqref{you} to discard the contribution of $T(q,\beta)$. A simple
solution is to assume that all the elements of our sequence are further
larger than $\sqrt{N}$, which is readily granted by assuming that~$M\ge \sqrt{N}$.

\begin{thm}
  \label{mainBourgainplus}
  There exists a constant $C>0$ such that the following holds.
  Assume $N\ge 10^{6}$, $M\ge\sqrt{N}$ and let $\ell\ge2$. We have
  \begin{equation*}
    \biggl(\int_0^1
    \biggl|\sum_{\substack{M+1\le p\le M+N}}
    \mkern-18mu{}u_p e(p\alpha)\biggr|^{\ell}d\alpha
    \biggr)^{1/\ell}
    \mkern-14mu
    \le C
    \sqrt{\frac{N/\log N}{1+R}}
    \biggl(\int_0^1
    \biggl|\sum_{M+1\le p\le M+N}\mkern-22mu
    e(p\alpha)\biggr|^{\ell}d\alpha
    \biggr)^{1/\ell}
  \end{equation*}
  as soon as $\sum_{M+1\le p\le
      M+N}|{}u_p|^2\le \sum_{M+1\le p\le
      M+N}1=R$.
\end{thm}
\noindent
So the uniform Hardy-Littlewood majorant property holds for primes in
the interval $[M+1,M+N]$ provided the number of such primes is
$\gg N/\log N$.
\begin{proof}[Proof of Theorem~\ref{mainBourgainplus}]
  We set
  \begin{equation}
    \label{eq:14}
    S(\alpha)=\sum_{M+1\le p\le M+N}e(p\alpha).
  \end{equation}
  We can assume that $S(0)\ge 1$.
  On following the proof of Theorem~\ref{mainBourgain}, we readily
  reach, when $|\beta|\le 1/(7N)$,
  \begin{equation}
    \label{eq:15}
    \sum_{a\mode q}
    \Bigl|S\Bigl(\frac aq+\beta\Bigr)\Bigr|^{\ell}
    \ge\frac{\mu^2(q)}{\varphi(q)^{\ell-1}}
    S(0)^\ell.
  \end{equation}
  This again leads to
  \begin{equation*}
    \int_0^1 |S(\alpha)|^\ell d\alpha
    \ge \frac{1-\sqrt{N}^{2-\ell}}{\ell-2}\frac{2}{7N}S(0)^\ell.
  \end{equation*}
  Corollary~\ref{mainCorbis} gives us with $\ell=2+h$
  \begin{align*}
    \int_0^1
    \biggl|\sum_{M+1\le p\le M+N}\mkern-16mu{}u_p e(p\alpha)\biggr|^{\ell}
    d\alpha
    &\le
    14000\frac{2^h+1/h}{N}
      \biggl(\frac{2N}{\log N}S(0)\biggr)^{\ell/2}
    \\&\le
    10^5
    \frac{h2^h+1}{1-\sqrt{N}^{-h}}\biggl(\frac{2N}{S(0)\log N}\biggr)^{\ell/2}
    \int_0^1 |S(\alpha)|^\ell d\alpha.
  \end{align*}
  The factor $(\frac{h2^h+1}{1-\sqrt{N}^{-h}})^{1/\ell}$ is bounded
    when
    $h\in[1/\log N,\infty)$. We treat separately the case
    $h\in[0,1/\log N]$.
\end{proof}

\noindent
Theorem~\ref{Extensionprecise} and
Corollary~\ref{Extensionprecisebis} go through with no
modifications, and are in this manner closer to~\eqref{oldram}.
\begin{thm}
  \label{Extensionpreciseplus}
  Let $\Xcal$ be a $\delta$-well spaced subset of
  $\mathbb{R}/\mathbb{Z}$ and $N\ge1000$.
  Let $({}u_p)_{M+1\le p\le M+N}$ be a sequence of complex numbers.
  We have
  \begin{equation*}
    \sum_{x\in\Xcal}
    \biggl|\sum_{M+1\le p\le M+N}\mkern-10mu{}u_p e(xp)\biggr|^2
    \le
     280 \frac{N+\delta^{-1}}{\log N}\sum_{M+1\le p\le M+N}\mkern-10mu|{}u_p|^2
    \log(2|\Xcal|).
  \end{equation*}
\end{thm}

\begin{cor}
  \label{Extensionprecisebisplus}
  Let $N\ge1000$ and $Q_0\in[2, \sqrt{N}]$.
  Let $({}u_p)_{M+1\le p\le M+N}$ be a sequence of complex numbers.
  We have
  \begin{equation*}
    \sum_{q\le Q_0}\sum_{a\mode q}
    \max_{|\alpha-\frac{a}{q}|\le \frac{1}{qQ_0}}
    \biggl|\sum_{M+1\le p\le M+N}\mkern-12mu{}u_p e(\alpha p)\biggr|^2
    \mkern-3mu
    \le
     1200 \frac{N\log Q_0}{\log N}\sum_{M+1\le p\le M+N}\mkern-12mu|{}u_p|^2
    .
  \end{equation*}
\end{cor}

\newpage

\bibliographystyle{ijmart}


\providecommand{\MR}{\relax\ifhmode\unskip\space\fi MR }
\providecommand{\MRhref}[2]{%
  \href{http://www.ams.org/mathscinet-getitem?mr=#1}{#2}
}
\providecommand{\href}[2]{#2}

\end{document}